\documentclass[12pt,a4paper,reqno]{amsart}
\usepackage{latexsym}
\usepackage{amssymb}
\usepackage{enumitem}

\def \zmontar{\buildrel}

\def \za{\alpha}

\def \ze{\varepsilon}

\def \zl{\lambda}
\def \zm{\mu}

\def \zp{\pi}

\def \zr{\rho}

\def \zt{\tau}

\def \zf{\varphi}

\def \zq{\psi}
\def \zw{\omega}

\def \zF{\Phi}

\def \zlma{\ell}

\def \zsu{\sum}

\def \zin{\cap}

\def \zun{\cup}
\def \zung{\bigcup}

\def \zdi{\oplus}

\def \zmm{\pm}

\def \zpu{\cdot}
\def \zpor{\times}
\def \zci{\circ}

\def \zmei{\leq}
\def \zmai{\geq}
\def \zco{\subset}

\def \zpe{\in}

\def \zeq{\equiv}

\def \znoi{\neq}

\def \zpar{\partial}
\def \zinf{\infty}
\def \zva{\emptyset}

\def \zfl{\rightarrow}

\def \zbv{\mid}

\def \z/{\over}

\newtheorem{theorem}{Theorem}
\newtheorem*{theorem*}{Theorem}
\newtheorem{lemma}{Lemma}
\newtheorem{corollary}{Corollary}
\newtheorem{proposition}{Proposition}

\newtheorem{remark}{Remark}

\newcommand{\Aut}{\operatorname{Aut}}
\newcommand{\Diff}{\operatorname{Diff}}

\title{Finite $C^{\zinf}$-actions are described by one
vector field}

\author{F.J.~Turiel}
\address[F.J.~Turiel]{
Departamento de {\'A}lgebra, Geometr{\'\i}a y Topolog{\'\i}a,
Facultad de Ciencias,
Campus de Teatinos, s/n,
29071-M{\'a}laga, Spain}
\email[F.J.~Turiel]{turiel@uma.es}

\author{A.~Viruel}
\address[A.~Viruel]{
Departamento de {\'A}lgebra, Geometr{\'\i}a y Topolog{\'\i}a,
Facultad de Ciencias,
Campus de Teatinos, s/n,
29071-M{\'a}laga, Spain} 
\email[A.~Viruel]{viruel@uma.es}

\thanks{Authors are partially supported by
MEC-FEDER grant MTM2010-18089, and JA grants FQM-213 and P07-FQM-2863}

\begin{document}

\begin{abstract}
In this work one shows that given a connected $C^{\zinf}$-manifold $M$
of dimension $\zmai 2$ and a finite subgroup $G\zco \Diff(M)$,
there exists a complete vector field $X$ on $M$ such that its
automorphism group equals $G\zpor \mathbb{R}$ where the factor
$\mathbb{R}$ comes from the flow of $X$.
\end{abstract}

\maketitle

\section{Introduction}

This work fits within the framework of the so called \emph{Inverse Galois Problem}: working in a category $\mathcal{C}$ and
given a group $G$, decide whether or not there exists an object $X$ in $\mathcal{C}$ such that $Aut_{\mathcal{C}}(X)\cong G$.

This metaproblem has been addressed by researchers in a wide range of situations from Algebra \cite{BW} and
Combinatorics \cite{F}, to Topology \cite{CV2}. In the setting of Differential Geometry, Kojima shows that any
finite group occurs as $\pi_0(\Diff(M))$ for some closed $3$-manifold
$M$ \cite[Corollary page 297]{Kojima}, and more recently Belolipetsky and Lubotzky \cite{BL} have proven that for
every $m \geq 2$, every
finite group is realized as the full isometry group of some compact hyperbolic $m$-manifold, so extending previous
results of Kojima \cite{Kojima} and Greenberg \cite{Greenberg}.

Here we consider automorphisms of vector fields. Although it is obvious that the automorphism group of a vector field
is never finite, we show that a given finite group of diffeomorphisms can be determined by a vector field. More precisely:

\begin{theorem*}
Consider a connected $C^{\zinf}$ manifold $M$ of dimension $m\zmai 2$
 and a finite subgroup $G$ of diffeomorphisms of $M$. Then there exists a complete $G$-invariant vector field $X$ on $M$, such that the map
$$\begin{array}{ccc}
G\zpor{\mathbb R}&\rightarrow& \Aut(X)\\
(g,t)&\mapsto&g\zci\zF_{t}
\end{array}$$ is a
group isomorphism, where $\zF$  and $\Aut(X)$ denote the flow and the group of automorphisms of $X$ respectively.
\end{theorem*}

Recall that, for any $m\zmai 2$, every finite group $G$ is a
quotient of the fundamental group of some compact, connected $C^{\zinf}$-manifold $M'$ of dimension $m$. Therefore $G$ can be regarded as the group of desk transformations of a connected covering $\zp:M\zfl M'$ and $G\zmei
\Diff(M)$. Consequently the result above solves the Galois
Inverse Problem for vector fields. Thus:

\begin{corollary}
Let $G$ be a finite group and $m\geq 2$, then there exists a connected
$C^{\zinf}$-manifold $M$ of dimension $m$ and a vector field $X$ on $M$ such that $\pi_0(\Aut(X))\cong G$.
\end{corollary}

Our results fit into the $C^{\zinf}$ setting, but it seems
interesting to study the same problem for other kind of manifolds
and, among them, the topological ones. Namely: given a finite
group ${\tilde G}$ of homeomorphisms of a connected topological manifold
${\tilde M}$ prove, or disprove, the existence of a continuous
action ${\tilde\zF}\colon{\mathbb R}\zpor{\tilde M}\zfl{\tilde M}$
such that:

\begin{enumerate}[label={\rm (\arabic{*})}]
\item ${\tilde\zF}_{t}\zci g=g\zci{\tilde\zF}_{t}$ for any $g\zpe{\tilde G}$ and
$t\zpe{\mathbb R}$.

\item If $f$ is a homeomorphism of ${\tilde M}$ and ${\tilde\zF}_{s}\zci
f=f\zci{\tilde\zF}_{s}$ for every $s\zpe{\mathbb R}$, then
$f=g\zci{\tilde\zF}_{t}$ for some $g\zpe{\tilde G}$ and
$t\zpe{\mathbb R}$ that are unique.
\end{enumerate}

This work, reasonably self-contained, consists of five sections,
the first one being the present Introduction. The others are
organized as follows. In Section \ref{sec-2}  some general
definitions and classical results are given. Section \ref{sec-3} is
devoted to the main result of this work (Theorem \ref{thm-1}) and
its proof. The extension of Theorem \ref{thm-1} to manifolds with non-empty boundary is addressed in Section \ref{sec-4}. The manuscript ends with an Appendix where a technical result needed in Section \ref{sec-4} is proven.

For the general questions on Differential Geometry the reader is
referred to \cite{KN} and for those on Differential Topology to
\cite{HI}.

\section{Preliminary notions}\label{sec-2}

Henceforth all structures and objects considered are real
$C^{\zinf}$ and manifolds without boundary, unless another thing
is stated. Given a vector field $Z$ on a $m$-manifold $M$ the
group of automorphisms of $Z$, namely $\Aut(Z)$, is the subgroup
of diffeomorphisms of $M$ that preserve $Z$, that is
$$\Aut(Z)=\{f\in \Diff(M): f_*(Z(p))=Z(f(p)) \text{ for all }p\in
M\}.$$

On the other hand, recall that a {\it regular trajectory} is the
trace of a non-constant maximal integral curve. Thus any regular
trajectory is oriented by the time in the obvious way and, if it
is not periodic, its points are completely ordered. As usual, a
{\it singular trajectory} is a singular point of $Z$.

If $Z(p)=0$ and $Z'$ is another vector field defined around $p$
then $[Z',Z](p)$ only depends on $Z'(p)$; thus the formula
$Z'(p)\zfl [Z',Z](p)$ defines an endomorphism of $T_{p}M$ called
{\it the linear part of $Z$ at $p$}. For the purpose of this work,
we will say that $p\zpe M$ is a {\it source} (respectively a {\it
sink}) of $Z$ if $Z(p)=0$ and its linear part at $p$ is the
product of a positive (negative) real number by the identity on
$T_{p}M$.

A point $q\zpe M$ is called a {\it rivet} if
\begin{enumerate}[label={\rm (\alph{*})}]
\item $q$ is an isolated singularity of $Z$,

\item\label{(II)} around $q$ one has $Z=\zq{\tilde Z}$ where $\zq$ is
a function and ${\tilde Z}$ a vector field with ${\tilde
Z}(q)\znoi 0$.
\end{enumerate}

Note that by \ref{(II)}, a rivet is the $\zw$-limit of exactly one
regular trajectory, the $\za$-limit of another one and an isolated
singularity of index zero.

Consider a singularity $p$ of $Z$; let $\zl_{1},\ldots,\zl_{m}$ be
the eigenvalues of the linear part of $Z$ at $p$ and
$\zm_{1},\ldots,\zm_{k}$ the same eigenvalues but only taking
each of them into
account once  regardless of its multiplicity. Assume
that $\zm_{1},\ldots,\zm_{k}$ are rationally independent; then
$\zl_{j}-\zsu_{\zlma=1}^{m}i_{\zlma}\zl_{\zlma}\znoi 0$ for any
$j=1,\ldots,m$ and any non-negative integers $i_{1},\ldots,i_{m}$
with $\zsu_{\zlma=1}^{m}i_{\zlma}\zmai 2$, and a theorem of
linearization by Sternberg (see \cite{SST} and \cite{RRO})  shows
the existence of coordinates $(x_{1},\ldots,x_{m})$ such that $p\zeq
0$ and $Z=\zsu_{j=1}^{m}\zl_{j}x_{j}\zpar/\zpar x_{j}$. That is
the case of sources ($\zl_{1}=\ldots=\zl_{m}>0$) and sinks
($\zl_{1}=\ldots=\zl_{m}<0$).

By definition, the {\it outset (or unstable manifold) $R_p$ of a
source} $p$ will be the set of all points $q\zpe M$ such that the
$\za$-limit of its $Z$-trajectory equals $p$. One has:

\begin{proposition}\label{prop-1}
Let $p$ be a source of a complete vector field $Z$. Then $R_p$ is
open and there exists a diffeomorphism from $R_p$ to ${\mathbb
R}^{m}$ that sends $p$ to the origin and $Z$ to
$a\zsu_{j=1}^{m}x_{j}\zpar/\zpar x_{j}$ for some $a\zpe{\mathbb
R}^{+}$. In other words, there exist coordinates
$(x_{1},\ldots,x_{m})$, whose domain $R_p$ is identified to ${\mathbb
R}^{m},$ such that $p\zeq 0$ and
$Z=a\zsu_{j=1}^{m}x_{j}\zpar/\zpar x_{j}$, $a\zpe{\mathbb R}^{+}$.
\end{proposition}

Indeed, let $\zF_t$ be the flow of $Z$; consider coordinates
$(y_{1},\ldots,y_{m})$ such that $p\zeq 0$ and
$Z=a\zsu_{j=1}^{m}y_{j}\zpar/\zpar y_{j}$. Up to dilation and with
the obvious identifications, one may suppose that $S^{m-1}$ is
included in the domain of these coordinates. Then
$R_{p}=\{\zF_{t}(y)\zbv t\zpe{\mathbb R}, y\zpe S^{m-1}\}\zun\{
0\}$ and it suffices to send the origin to the origin and each
$\zF_{t}(y)$ to $e^{at}y$ for constructing the required
diffeomorphism.

\begin{remark}\label{rem-1}
{\rm Observe that $R_{p}\zin R_{q}=\zva$ when $p$ and $q$ are
different sources of $Z$.}
\end{remark}

Given a regular trajectory $\zt$ of $Z$ with $\za$-limit a source
$p$, by the {\it linear $\za$-limit of} $\zt$ one means the (open
and starting at the origin) half-line in the vector space $T_{p}M$
that is the limit, when $q\zpe\zt$ tends to $p$, of the half-line
in $T_{q}M$ spanned by $Z(q)$. From the local model around $p$
follows the existence of this limit; moreover if $Z$ is multiplied
by a positive function the linear $\za$-limit does not change.

By definition, a {\it chain} of $Z$ is a finite and ordered
sequence of two or more different regular trajectories, each of
them called a {\it link}, such that:

\begin{enumerate}[label={\rm (\alph{*})}]
\item The $\za$-limit of the first link is a source.

\item The $\zw$-limit of the last link is not a rivet.

\item Between two consecutive links the $\zw$-limit of the
first one equals the $\za$-limit of the second one. Moreover this
set consists in a rivet.
\end{enumerate}

The {\it order of a chain} is the number of its links and its {\it
$\za$-limit} and {\it linear $\za$-limit} those  of its first
link.

For sake of simplicity, here countable includes the finite case as
well. One says that a subset $Q$ of $M$ {\it does not exceed
dimension $\zlma$}, or it {\it can be enclosed in dimension
$\zlma$}, if there exists a countable collection $\{
N_{\zl}\}_{\zl\zpe L}$ of submanifolds of $M$, all of them of
dimension $\zmei\zlma$, such that $Q\zco \zung_{\zl\zpe
L}N_{\zl}$. Note that the countable union of sets whose dimension
do not exceed dimension $\zlma$ does not exceed dimension $\zlma$
too. On the other hand, if $\zlma<m$ then $Q$ has measure zero so
empty interior.

Given a $m$-dimensional real vector space $V$, a family
$\mathcal{L}= \{L_{1},\ldots,L_{s}\}$, $s\zmai m$, of half-lines of
$V$ is named {\it in general position} if any subfamily of
$\mathcal{L}$ with $m$ elements spans $V$.

Now consider a finite group $H\zco GL(V)$ of order $k$. A family
${\mathcal L}$ of half-lines of $V$ is named a {\it control family
with respect to $H$} if:
\begin{enumerate}[label={\rm (\alph{*})}]
\item $h(L)\zpe{\mathcal L}$ for any $h\zpe H$ and $L\zpe{\mathcal
L}$.

\item There exists a family ${\mathcal L}'$ of ${\mathcal L}$ with
$km+1$ elements, which is in general position, such that
$H\zpu{\mathcal L}'=\{h(L)\zbv h\zpe H,L\zpe{\mathcal L}'\}$
equals ${\mathcal L}$.
\end{enumerate}

\begin{lemma}\label{lem-1}
Let ${\mathcal L}$ be a control family with respect to $H$ and
$\zf$ an element of $ GL(V)$. If $\zf$ sends each orbit of the
action of $H$ on ${\mathcal L}$ into itself, then $\zf=ah$ for
some $a\zpe{\mathbb R}^{+}$ and $h\zpe H$.
\end{lemma}

Indeed, as for every $L\zpe{\mathcal L}'$ there is $h'\zpe H$ such
that $\zf(L)=h'(L)$, there exist a subfamily ${\mathcal
L}''=\{L_{1},\ldots,L_{m+1}\}$ of ${\mathcal L}'$ and a $h\zpe H$
such that $\zf(L_{j})=h(L_{j})$, $j=1,\ldots,m+1$. Therefore
$h^{-1}\zci\zf$ sends $L_j$ into $L_j$, $j=1,\ldots,m+1$, and
because ${\mathcal L}''$ is in general position $h^{-1}\zci\zf$
has to be a multiple of the identity. Since every $L_j$ is a
half-line this multiple is positive.

\section{The main result}\label{sec-3}

This section is devoted to prove the following result on finite
groups of diffeomorphisms of a connected manifold.
\begin{theorem}\label{thm-1}
Consider a connected manifold $M$ of dimension $m\zmai 2$ and a
finite group $G\zco\Diff(M)$. Then there exists a complete vector
field $X$ on $M$, which is $G$-invariant, such that the map
$$(g,t)\zpe G\zpor{\mathbb R}\zfl g\zci\zF_{t}\zpe\Aut(X)$$ is a
group isomorphism, where $\zF$ denotes the flow of $X$.
\end{theorem}

Consider a Morse function $\zm\colon M\zfl{\mathbb R}$ that is
$G$-invariant, proper and non-negative, whose existence is assured
by a result of Wasserman (see the remark of page 150 and the proof
of Corollary 4.10 of \cite{WA}). Denote by $C$ the set of its
critical points, which is closed, discrete (that is without
accumulation points in $M$) so countable. As $M$ is paracompact, there
exists a locally finite family $\{A_{p}\}_{p\zpe C}$ of disjoint
open set such that $p\zpe A_{p}$ for every $p\zpe C$.

\begin{lemma}\label{lem-2}
There exists a $G$-invariant Riemannian metric ${\tilde g}$ on $M$
such that if $J(p)\colon T_{p}M\zfl T_{p}M$, $p\zpe C$, is defined
by $H(\zm)(p)(v,w)={\tilde g}(p)(J(p)v,w)$, where $H(\zm)(p)$ is
the hessian of $\zm$ at $p$, then:
\begin{enumerate}[label={\rm (\arabic{*})}]
\item If $p$ is a maximum or a minimum then $J(p)$ is a multiple
of the identity.

\item If $p$ is a saddle, that is $H(\zm)(p)$ is not definite,
then the eigenvalues of $J(p)$ avoiding repetitions due to the
multiplicity are rationally independent.
\end{enumerate}
\end{lemma}

\begin{proof} We start constructing a 'good' scalar product on
each $T_{p}M$, $p\zpe C$. If $p$ is a minimum [respectively
maximum] one takes $H(\zm)(p)$ [respectively $-H(\zm)(p)$]. When
$p$ is a saddle consider a scalar product $\langle\, ,\,\rangle$ on
$T_{p}M$ invariant by the linear action of the isotropy group
$G_p$ of $G$ at $p$. In this case as $J(p)$ is $G_p$-invariant (of
course here $J(p)$ is defined with respect to $\langle\, ,\,\rangle$),
$T_{p}M=\zdi_{j=1}^{k}E_{j}$ and $J(p)_{\zbv E_{j}}=a_{j}Id_{\zbv
E_{j}}$ where each $E_j$ is $G_p$-invariant, $a_{j}\znoi 0$,
$\langle E_{j},E_{\zlma}\rangle=0$ and $a_{j}\znoi a_{\zlma}$ if $j\znoi\zlma$.

Besides one may suppose $a_{1},\ldots,a_{k}$ rationally independent
by taking, if necessary, a new scalar product $\langle\, ,\,\rangle'$
such that $\langle E_{j},E_{\zlma}\rangle'=0$ when $j\znoi\zlma$ and
$\langle\, ,\,\rangle'_{\zbv E_{j}}=b_{j}\langle\, ,\,\rangle_{\zbv E_{j}}$ for
suitable scalars $b_{1},\ldots,b_{k}$.

In turns this family of scalar products on $\{T_{p}M\}_{p\zpe C}$
can be construct $G$-invariant. Indeed, this is obvious for maxima
and minima since $\zm$ is $G$-invariant. On the other hand, if
$C'\zco C$ is a $G$-orbit consisting of saddles take a point $p$
in $C'$, endow $T_{p}M$ with a 'good' scalar product and extend to
$C'$ by means of the action of $G$.

It is easily seen, through the family $\{A_{p}\}_{p\zpe C}$, that
of all these scalar products on $\{T_{p}M\}_{p\zpe C}$ extend to a
Riemannian metric ${\tilde g}$ on $M$. Finally, if ${\tilde g}$ is
not $G$-invariant consider $\zsu_{g\zpe G}g^{*}({\tilde g})$.
\end{proof}

Let $Y$ be the gradient vector field of $\zm$ with respect to some
Riemannian metric ${\tilde g}$ as in Lemma \ref{lem-2}. We will
assume that $Y$ is complete by multiplying, if necessary, ${\tilde
g}$ by a suitable $G$-invariant positive function (more exactly by
$e^{(Y\zpu\zr)^{2}}$ where $\zr$ is a $G$-invariant proper
function). Since $\zm$ is non-negative and proper, the $\za$-limit
of any regular trajectory of $Y$ is a local minimum or a saddle of
$\zm$, whereas its $\zw$-limit is empty, a local maximum or a
saddle of $\zm$.

Now $Y^{-1}(0)=C$ and, by the Sternberg's Theorem, around each
$p\zpe C$ (note that the linear part of $Y$ at $p$ equals
$J(p)\colon T_{p}M\zfl T_{p}M$ defined in Lemma \ref{lem-2}) there
exist a natural $1\zmei k\zmei m-1$ and coordinates
$(x_{1},\ldots,x_{m})$ such that $p\zeq 0$ and
$Y=\zsu_{j=1}^{m}\zl_{j}x_{j}\zpar/\zpar x_{j}$ where
$\zl_{1},\ldots,\zl_{k}>0$ and $\zl_{k+1},\ldots,\zl_{m}<0$, or
$Y=a\zsu_{j=1}^{m}x_{j}\zpar/\zpar x_{j}$ where $a>0$ if $p$ is a
source (that is a minimum of $\zm$) and $a<0$ if $p$ is a sink (a
maximum of $\zm$.)

Let $I$ be the set of local minima of $\zm$, that is the set of
sources of $Y$, and $S_{i}$, $i\zpe I$, the outset of $i$ relative
to $Y$.  Obviously $G$ acts on the set $I$.

\begin{lemma}\label{lem-3}
In $M$ the family $\{ S_{i}\}_{i\zpe I}$ is locally finite and the
set $\zung_{i\zpe I}S_{i}$ dense.
\end{lemma}

\begin{proof}
First notice that $\zm(S_{i})$ is low bounded by $\zm(i)$. But $I$
is a discrete set and $\zm$ a non-negative proper Morse function,
so in every compact set $\zm^{-1}((-\zinf,a])$ there are only a
finite number of elements of $I$. Therefore $\zm^{-1}((-\zinf,a])$
and of course $\zm^{-1}(-\zinf,a)$ only intersect a finite number
of $S_i$. Finally, observe that $M=\zung_{a\zpe\mathbb
R}\zm^{-1}(-\zinf,a)$.

If the $\za$-limit of the $Y$-trajectory of $q$ is a saddle $s$,
with the local model given above there exists $t\zpe {\mathbb Q}$
such that $\zF_{t}(q)$ is close to $s$ and
$x_{k+1}(\zF_{t}(q))=\ldots=x_{m}(\zF_{t}(q))=0$. Since the
submanifold given by the equations $x_{k+1}=\ldots=x_{m}=0$ has
dimension $\zmei m-1$ and ${\mathbb Q}$  and the set of saddles
are countable, it follows that the set of points coming from a
saddle may be enclosed in dimension $m-1$ and its complementary,
that is $\zung_{i\zpe I}S_{i}$, has to be dense.
\end{proof}

The vector field $Y$ has no rivets since all its singularities are
isolated with index $\zmm 1$, therefore it has no chain; moreover
the regular trajectories are not periodic.

For each $i\zpe I$, let $ {\mathcal L} _{i}$ be a control family on $T_{i}M$ with
respect to the  action of the isotropy group $G_i$ of $G$ at $i$, such that if $g(i)=i'$
then $g$ transforms $ {\mathcal L} _{i}$ in $ {\mathcal L} _{i'}$ . These families can be constructed as
follows: for every orbit of the action of $G$ on $I$ choose a point $i$ and $k_{i}m+1$
different half-lines in general position, where $k_i$ is the order of $G_i$; now $G_i$-saturate
this first family for giving rise to $ {\mathcal L} _{i}$. For other points $i'$ in the same orbit
choose $g\zpe G$ such that $g(i)=i'$ and move $ {\mathcal L} _{i}$ to $i'$ by means of $g$.

Let  $ {\mathcal L} $ be the set of all elements of $ {\mathcal L} _{i}$, $i\zpe I$. By Proposition
\ref{prop-1} each element of  $ {\mathcal L} $  is the linear $\za$-limit of just one trajectory of $Y$;
let ${\mathcal T}$ be the set of such trajectories. Clearly $G$ acts on ${\mathcal T}$, since $Y$
and ${\mathcal L}$ are $G$-invariant, and the set of orbits of this action is countable. Therefore this
last one can be regarded as a family $\{P_{n}\}_{n\zpe {\mathbb N}' }$ where
$ {\mathbb N}' \zco {\mathbb N}- \{0,1\}$, each $P_n$ is a $G$-orbit
and $P_{n}\znoi P_{n'}$ if $n\znoi n'$.

In turns, in each $T\zpe P_n$ one may choose $n-1$ different points in such a way that if $T'=g(T)$ then
$g$ sends the points considered in $T$ to those of $T'$. Denoted by $W_n$ the set of all points chosen
in the trajectories of $P_n$.

Since $\{S_{i}\}_{i\zpe I}$ is locally finite (Lemma \ref{lem-3}),
the set $W=\zung _{n \zpe{\mathbb N}' }W_n$ is discrete,  countable, closed and $G$-invariant.
Therefore there exists a $G$-invariant function
$\zq:M\zfl \mathbb R$, which is non negative and bounded,
such that $\zq^{-1}(0)= W$. Set $Y=\zf Z$. One has:

\begin{enumerate}[label={\rm (\alph{*})}]
\item $G$ is a subgroup of $\Aut(X)$.

\item $X^{-1}(0)= Y^{-1}(0)\zun W$, the rivets of $X$ are just the points of $W$ and $X$
has no periodic regular trajectories.

\item $X$ and $Y$  have the same sources, sinks and saddles.
Moreover if $R_i$ , $i\zpe I$,  is the $X$-outset of $i$ , then $R_{i}\zco S_{i}$ and
$\zung_{i\zpe I}(S_{i}-R_{i})\zco \zung_{T\zpe P_{n},n\zpe\mathbb{N}'}T$,
so $\{R_{i}\}_{i\zpe I}$ is locally finite and $\zung_{i\zpe I}R_{i}$ is dense.

\item Let $C_T$, $T\zpe P_n$, $n\zpe {\mathbb N}' $, be the family of $X$-trajectories of
$T-W$ endowed with the order induced by that of $T$ as
$Y$-trajectory. Then $C_{T}$ is a chain of $X$ of order $n$ whose
rivets are the points of $T\zin W$
and whose $\za$-limit and linear $\za$-limit are those of $T$.
Besides $C_T$, $T\zpe P_n$, are the only chain of $X$ of order $n$.
\end{enumerate}

As each $P_n$ is a $G$-orbit in $ {\mathcal T} $, the group $G$ acts on the set of
chains of $X$ and every $\{C_{T}\zbv T\zpe P_{n}\}$ is an orbit. Thus $G$ acts transitively
on  the set of $\za$-limit and on that of linear $\za$-limit of the chains  $C_T$, $T\zpe P_n$.
Recall that:

\begin{lemma}\label{lem-4}
Any map $\zf:{\mathbb R}^{k}\zfl{\mathbb R}^{s}$ such that
$\zf(ay)=a\zf(y)$, for all $(a,y)\zpe{\mathbb R}^{+}\zpor{\mathbb
R}^{k}$, is linear.
\end{lemma}

\begin{remark}\label{rem-2}
{\rm As it is well known, the foregoing lemma does not hold for
continuous maps (in this work maps are $C^{\zinf}$ unless another
thing is stated).}
\end{remark}

\begin{proposition}\label{prop-2}
Given $f\zpe\Aut(X)$ and $i\zpe I$ there exists $(g,t)\zpe G\zpor  {\mathbb R}$
such that $f=g\zci \zF_t$ on $R_i$.
\end{proposition}

\begin{proof} Consider $n\zpe {\mathbb N}' $ such that $i$ is the $\za$-limit of some
chain of order $n$. Then $f(i)$ is the $\za$-limit of some
chain of order $n$ and there exists $g\zpe G$ such that $g(i)=f(i)$; therefore
$(g^{-1}\zci f)(i)=i$, which reduces the problem, up to change of notation, to consider the case where $f(i)=i$.

Note that every $L\zpe {\mathcal L} _{i}$ is the linear $\za$-limit of some $T\zpe {\mathcal T}$,
so the linear $\za$-limit of $C_T$; moreover $ {\mathcal L} _{i}$ is the family of linear $\za$-limit
of all chains starting at $i$. As $f$ sends chains starting at $i$ into chains starting at $i$
because $f$ is an automorphism of $X$, follows that $f_{*}(i)$ sends $ {\mathcal L} _{i}$ into itself.

On the other hand, since for any $T\zpe P_n$ one has $f(C_{T})=C_{T'}$ where $T'$ belongs to $P_n$ as well,
it has to exists $h\zpe G$ that sends the linear $\za$-limit of $C_T$ to the linear $\za$-limit of $C_{T'}$.
But both chains start at $i$ so $h\zpe G_i$, which implies that $f_{*}(i)$ preserves each orbit of the
action of $G_i$ on $ {\mathcal L} _{i}$. From Lemma \ref{lem-1} follows that $f_{*}(i)=ch_{*}(i)$
with $c>0$ and $h\zpe G_i$. Therefore considering $h^{-1}\zci f$ we may suppose, up to a new
change of notation, that $f_{*}(i)=cId$, $c>0$.

Now Proposition \ref{prop-1} allows us to regard $f$ on $R_{i}$ as a map
 $\zf:{\mathbb R}^{m}\zfl{\mathbb R}^{m}$ that preserves the vector field
$X=a\zsu_{j=1}^{m}x_{j}\zpar/\zpar x_{j}$, $a\zpe{\mathbb R}^{+}$.
But this last property implies that $\zf(bx)=b\zf(x)$ for any
$b\zpe{\mathbb R}^{+}$ and $x\zpe{\mathbb R}^{m}$; therefore $\zf$
is linear (Lemma \ref{lem-4}). Since $f_{*}(i)=cId$ one has $\zf=cId$, $c>0$;
that is to say $\zf$ and $f_{\zbv R_{i}}$ equal $\zF_t$ for some $t\zpe {\mathbb R}$.
\end{proof}

Given $f\zpe \Aut(X)$, consider a family $\{(g_{i},t_{i})\}_{i\zpe}$ of
elements of $G\zpor {\mathbb R}$ such that $f=g_{i}\zci\zF_{t_{i}}$ on each $R_i$.
We will show that $f=g\zci\zF_t$ for some $g\zpe G$, $t\zpe {\mathbb R}$.
\begin{lemma}\label{lem-5}
If all $g_i$ are equal then all $t_i$ are equal too.
\end{lemma}

\begin{proof} The proof reduces to the case where all $g_{i}=e_{G}$ (neutral element
of $G$) by composing $f$ on the left with a suitable element of $G$.
Obviously $f=\zF_{t_{i}}$ on ${\overline R}_i$.

Assume that the set of these $t_i$ has more than one
element. Fixed one of them, say $t$, set $D_1$ the union of
all ${\overline R}_{i}$ such that
$t_{i}=t$ and $D_2$ the union of all
${\overline R}_{i}$ such that $t_{i}\znoi t$.
Since $\{R_{i}\}_{i\zpe I}$ is locally finite and $\zung_{i\zpe I}R_{i}$ dense,
the family  $\{ {\overline R}_{i}\}_{i\zpe I}$ is locally finite too and
 $\zung_{i\zpe I}{ {\overline R}_{i}}=M$. Thus $D_1$ and $D_2$ are closed and $M=D_{1}\zun D_{2}$.
On the other hand if $p\zpe D_{1}\zin D_{2}$ then $\zF_{t}(p)=\zF_{t_{i}}(p)$
for some $t\znoi t_i$, so $\zF_{t-t_{i}}(p)=p$ and $X(p)=0$ since $X$ has
no periodic regular trajectories, which implies that $ D_{1}\zin D_{2}$ is countable. Consequently
$M-D_{1}\zin D_{2}$ is connected. But $M-D_{1}\zin
D_{2}=(D_{1}-D_{1}\zin D_{2})\zun(D_{2}-D_{1}\zin D_{2})$ where
the terms of this union are non-empty, disjoint and closed in
$M-D_{1}\zin D_{2}$, {\it contradiction}.
\end{proof}

Choose a $i_{0}\zpe I$. Composing $f$ on the left with a suitable element of $G$
we may assume $g_{i_{0}}=e_G$. On the other hand, $f$ sends each orbit of
the actions of $G$ on $I$ into itself because the points of every orbit
are just the starting points of the chains of order $n$ for some $n\zpe {\mathbb N}' $.
Thus $f$ equals a permutation on each orbit of $G$ in $I$ and there exists $\zlma>0$
such that $f^{\zlma}$ is the identity on these orbits; for instance $\zlma=r!$ where $r$ is the order of $G$.

Now suppose that $f^{\zlma}=h_{i}\zci\zF_{s_{i}}$ on $R_i$, $i\zpe
I$. Then $h_{i}\zpe G_i$. Since the order of $G_i$ divides that of
$G$ one has $f^{r\zlma}=\zF_{rs_{i}}$ on $R_i$. In short, there
exists a natural number $k>0$ such that  $f^{k}=\zF_{u_{i}}$ on
$R_i$, and by Lemma \ref{lem-5} one has $f^{k}=\zF_{u}$ on every
$R_i$ for some $u\zpe {\mathbb R}$.

In turns, composing $f$ with $\zF_{-u/k}$ we may assume, without lost
of generality, that  $f^{k}=Id$ on $M$.

On $R_{i_{0}}$ one has  $f^{k}=\zF_{kt_{i_{0}}}$, so $t_{i_{0}}=0$ and $f=Id$.
But $f$ spans a finite group of diffeomorphisms of $M$, which assure us
that $f$ is an isometry of some Riemannian metric $\hat g$ on  $M$. Recall that
isometries on connected manifolds are determined by the $1$-jet at
any point. Therefore from $f=Id$ on $R_{i_{0}}$ follows $f=Id$ on $M$.

In other words the map $(g,t)\zpe G\zpor {\mathbb R}\zfl g\zci\zF_{t}\zpe\Aut(X)$ is
an epimorphism. Now the proof of Theorem \ref{thm-1} will be finished showing that it is an injection.

Assume that $g\zci\zF_{t}=Id$ on $M$. As $g^{r}=e_G$ follows $\zF_{rt}=Id$
whence $t=0$ because $X$ has no periodic regular trajectories. Thus $g=e_G$.

\begin{remark}\label{rem-3}
{\rm From the proof of Theorem \ref{thm-1} above, follows that this theorem holds for
$X'=\zr X$ where $\zr\colon M\zfl {\mathbb R}$ is any  $G$-invariant positive
bounded function. Indeed, reason as before with $(\zr\zq)Y$ instead of $\zq Y$. }
\end{remark}

\section{Actions on manifolds with boundary}\label{sec-4}

Let $P$ be an $m$-manifold with non-empty boundary $\zpar P$. Set
$M=P-\zpar P$. First recall that there always exist a manifold
$\tilde P$ without boundary and a function ${\tilde\zf}:{\tilde
P}\zfl{\mathbb R}$ such that zero is a regular value of
${\tilde\zf}$ and $P$ diffeomorphic to
${\tilde\zf}^{-1}((-\zinf,0])$; so let us identify $P$ and
${\tilde\zf}^{-1}((-\zinf,0])$.

Now assume that $G$ is a finite subgroup of $\Diff(P)$, $P$ is
connected and $m\zmai 2$. Then $G$ sends $\zpar P$ to $\zpar P$
and $M$ to $M$; thus by restriction $G$ becomes a finite subgroup
of $\Diff(M)$.

Let $X'$ be a vector field as in the proof of Theorem \ref{thm-1} with respect
to $M$ and $G\zco \Diff(M)$.
By Proposition \ref{prop-A1} in the Appendix  (Section \ref{sec-A})
applied to $M$ and $X'$, there exists a bounded function
 $\zf\colon{\tilde P}\zfl{\mathbb R}$, which is positive on $M$ and vanishes  elsewhere, such that the
 vector field $\zf X'$ on $M$ prolongs by zero to a (differentiable) vector field on $ {\tilde P}$.

\begin{lemma}\label{lem-6}
For every $g\zpe G$ the vector field $X_g$ equal to $(\zf\zci g)X'$ on $M$ and vanishing elsewhere is differentiable.
\end{lemma}

\begin{proof} Obviously $X_g$ is smooth on $ {\tilde P}-\zpar P$. Now consider any $p\zpe \zpar P$.
As $g\colon P\zfl P$ is a diffeomorphism, there exist an open neighborhood $A$ of $p$
on $ {\tilde P}$ and a map $ {\hat g}\colon A\zfl {\tilde P}$ such that $ {\hat g}=g$ on $A\zin P$.
Shrinking $A$ allows to assume that $B= {\hat g}(A)$ is open,  $ {\hat g}\colon A\zfl B$ is
a diffeomorphism and $A-\zpar P$ has two connected components $A_{1},A_{2}$ with $A_{1}\zco M$
and $A_{2}\zco {\tilde P}-P$; note that
${\hat g}(A_{1})\zco M$, ${\hat g}(A_{2})\zco{\tilde P}-P$ and $ {\hat g}(A\zin\zpar P)\zco \zpar P$.

Thus $(X_{g})_{\zbv A}= {\hat g}_{*}^{-1}(X_{\zf})_{\zbv B}$ since $X'$ is $G$-invariant.
\end{proof}

On $P$ set $X=\zsu_{g\zpe G}X_g$. Then $X_{\zbv\zpar P}=0$ and  $X_{\zbv M}=\zr X'$
where $\zr=\zsu_{g\zpe G}(\zf_{\zbv M})\zci g$. Clearly $\zr\colon M\zfl {\mathbb R}$
is positive bounded and $G$-invariant, so by Remark \ref{rem-3} Theorem \ref{thm-1}
also holds for $X_{\zbv M}$. Moreover $X$ is complete on $P$.

If $f\colon P\zfl P$ belongs to $\Aut(X)$ then $f_{\zbv M}$ belongs to $\Aut(X_{\zbv M})$
and $f=g\zci\zF_t$ on $M$ and by continuity on $P$.
In other words, Theorem \ref{thm-1} also holds for any connected
manifold $P$, of dimension $\zmai 2$, with non-empty boundary.

\section{Appendix}\label{sec-A}

In this appendix we prove Proposition \ref{prop-A1} that
was needed in the foregoing section. First consider a family of compact sets
$\{K_{r}\}_{r\zpe\mathbb N}$ in an open set $A\zco{\mathbb
R}^{n}$, such that $K_{r}\zco{\zmontar \zci\over{K}}_{r+1}$,
$r\zpe\mathbb N$, and $\zung_{r\zpe\mathbb N}K_{r}=A$.

\begin{lemma}\label{lem-A1}
Given a family of positive continuous functions
$\{f_{r}:A\zfl{\mathbb R}\}_{r\zpe\mathbb N}$ there exists a
function $f:A\zfl{\mathbb R}$ vanishing on ${\mathbb R}^n-A$ and positive on
$A$ such that, whenever $r\zpe\mathbb N$, one has $f\zmei f_j$,
$0\zmei j\zmei r$, on $A-K_r$.
\end{lemma}
\begin{proof} One may assume $f_{0}\zmai f_{1}\zmai\ldots \zmai
f_{r}\zmai\ldots $ by taking $min\{f_{0},\ldots ,f_{r}\}$ instead of
$f_r$ if necessary. Consider  functions $\zf_{r}:{\mathbb
R}^{n}\zfl [0,1]\zco{\mathbb R}$, ${r\zpe\mathbb N}$, such that
each $\zf_{r}^{-1}(0)=K_{r-1}\zun({\mathbb R}^{n}-{\zmontar
\zci\over{K}}_{r+1})$ [as usual $K_{j}=\zva$ if $j\zmei -1$].

Let $D$ be a partial derivative operator. Multiplying each $f_r$
by some $\ze_{r}>0$ small enough allows to suppose, without loss
of generality, $\zf_{r}\zmei f_{r}/2$  on $A$ and $\zbv
D\zf_{r}\zbv \zmei 2^{-r}$ on ${\mathbb R}^{n}$ for any $D$ of
order $\zmei r$.

Set $f=\zsu_{r\zpe\mathbb N}\zf_{r}$. By the second condition on
functions $\zf_r$, whenever $\tilde D$ is a partial derivative
operator the series $\zsu_{r\zpe\mathbb N}{\tilde D}\zf_{r}$
uniformly converges on ${\mathbb R}^{n}$, which implies that $f$
is differentiable. On the other hand it is easily checked that
$f({\mathbb R}^{n}-A)=0$, $f>0$ on $A$ and $f\zmei f_{r}\zmei\ldots
\zmei f_{0}$ on $A-K_{r}$.
\end{proof}

One will say that a function defined around a point $p$ of a
manifold is {\it flat at $p$} if its $\zinf$-jet at this point
vanishes. Note that given a function $\zq$ on a manifold and a
function $\zt:{\mathbb R}\zfl [0,1]\zco {\mathbb R}$ flat at the
origin and positive on ${\mathbb R}-\{0\}$ (for instance
$\zt(t)=e^{-1/t^{2}}$ if $t\znoi 0$ and $\zt(0)=0$), then
$\zt\zci\zq$ is flat at every point of $(\zt\zci\zq)^{-1}
(0)=\zq^{-1} (0)$ and $Im(\zt\zci\zq)\zco [0,1]$.

\begin{lemma}\label{lem-A2}
Consider an open set $A$ of a manifold $M$ and a function
$f:A\zfl{\mathbb R}$. Then there exists a function
$\zf:M\zfl{\mathbb R}$ vanishing on $M-A$ and positive on $A$,
such that the function ${\hat f}:M\zfl{\mathbb R}$ given by ${\hat
f}=\zf f$ on $A$ and ${\hat f}=0$ on $M-A$ is differentiable.
\end{lemma}
\begin{proof} The manifold $M$ can be seen as a closed imbedded
submanifold of some ${\mathbb R}^{n}$. Let $\zp:E\zfl M$ be a
tubular neighborhood of $M$. If the result is true for
$\zp^{-1}(A)$ and $f\zci\zp:\zp^{-1}(A)\zfl{\mathbb R}$, by
restriction it is true for $A$ and $f$. In other words, it
suffices to consider the case of an open set $A$ of ${\mathbb
R}^{n}$.

We will say that a function $\zq:A\zfl{\mathbb R}$ is {\it neatly
bounded} if, for each point $p$ of the topological boundary of $A$
and any partial derivative operator $D$, there exists an open
neighborhood $B$ of $p$ such that $\zbv D\zq\zbv$ is bounded on
$A\zin B$. First assume that $f$ is neatly bounded. Let
$\zf:{\mathbb R}^{n}\zfl{\mathbb R}$ be a function that is
positive on $A$ and flat at every point of ${\mathbb R}^{n}-A$;
then $\zf$ satisfies Lemma \ref{lem-A2}.

Indeed, only the points $p\zpe({\bar A}-A)$ need to be examined.
Consider an natural $1\zmei j\zmei n$; since $j_{p}^{\zinf}\zf=0$
near $p$ one has
$\zf(x)=\zsu_{i=1}^{n}(x_{i}-p_{i}){\tilde\zf}_{i}(x)$ and from
the definition of partial derivative follows that $(\zpar{\hat
f}/\zpar x_{j})(p)=0$. Thus $\zpar{\hat f}/\zpar
x_{j}=(\zpar\zf/\zpar x_{j})f+\zf\zpar f/\zpar x_{j}$ on $A$ and
$\zpar{\hat f}/\zpar x_{j}=0$ on ${\mathbb R}^{n}-A$, which shows
that $f$ is $C^{1}$.

Since obviously the function $\zpar f/\zpar x_{j}$ is neatly
bounded and $\zpar\zf/\zpar x_{j}$ is flat on ${\mathbb R}^{n}-A$,
the same argument as before applied to $(\zpar\zf/\zpar x_{j})f$
and $\zf\zpar f/\zpar x_{j}$ shows that $f$ is $C^{2}$ and, by
induction, the differentiability of $f$.

Let us see the general case. On $A$ the
continuous functions $\zbv Df\zbv +1$, where $D$ is any partial
derivative operator, give rise to a countable family of continuous
positive functions $g_{0},\ldots ,g_{r},\ldots $. Let
$\{K_{r}\}_{r\zpe\mathbb N}$ be a collection of compact sets such
that $K_{r}\zco {\zmontar \zci\over{K}}_{r+1}$, ${r\zpe\mathbb
N}$, and $\zung_{r\zpe\mathbb N}K_{r}=A$. By Lemma \ref{lem-A1}
there exists a function $\zr:{\mathbb R}^{n}\zfl{\mathbb R}$
vanishing on ${\mathbb R}^{n}-A$ and positive on $A$ such that
$\zr\zmei g_{j}^{-1}$, $0\zmei j\zmei r$, on $A-K_r$,
${r\zpe\mathbb N}$.

For every $k\zpe\mathbb N$ let $\zl_{k}:{\mathbb R}\zfl{\mathbb
R}$ be the function defined by $\zl_{k}(t)=t^{-k}e^{-1/t}$ if
$t>0$ and $\zl_{k}(t)=0$ elsewhere. Then the function ${\tilde
f}=\zl_{0}(\zr/2)f$ is neatly bounded on $A$. Indeed, consider any
$p\zpe({\bar A}-A)$ and any partial derivative operator $D$. Then
$D{\tilde f}$ equals a linear combination, with constant
coefficients, of  products of some partial derivatives of $\zr$, a
function $\zr^{-k}e^{-2/\zr}= \zl_{k}(\zr)e^{-1/\zr}$ and some
partial derivative $D'f$. On the other hand, there always exists a
natural $\zlma$ such that $g_{\zlma}=\zbv D'f\zbv +1$. But near
$p$ one has $e^{-1/\zr}\zbv D'f\zbv\zmei\zr\zbv D'f\zbv\zmei\zr
g_{\zlma}\zmei 1$; therefore $D\tilde f$ is bounded close to $p$.

Finally, take a function ${\tilde\zf}:{\mathbb R}^{n}\zfl{\mathbb
R}$ positive on $A$ and flat at every point of ${\mathbb R}^{n}-A$
and set $\zf={\tilde\zf}\zl_{0}(\zr/2)$.
\end{proof}

\begin{proposition}\label{prop-A1}
Consider a vector field $X$ on an open set $A$ of a manifold $M$.
Then there exists a bounded function $\zf:M\zfl{\mathbb R}$, which
is positive on $A$ and vanishes on $M-A$, such that the vector
field ${\hat X}$ on $M$ defined by ${\hat X}=\zf X$ on $A$ and
${\hat X}=0$ on $M-A$ is differentiable.
\end{proposition}

\begin{proof} Regard $M$ as a closed imbedded submanifold of some
${\mathbb R}^{n}$; let $\zp:E\zfl M$ be a tubular neighborhood of
$M$. Then there exists a vector field $X'$ on $\zp^{-1}(A)$ such
that $X'=X$ on $A$ and, by restriction of the function, it
suffices to show our result for $X'$ and $\zp^{-1}(A)$. That is to
say, we may suppose, without loss of generality, that $A$ is an
open set of ${\mathbb R}^{n}$.

In this case on $A$ one has $X=\zsu_{j=1}^{n}f_{j}\zpar/\zpar
x_{j}$. Applying Lemma \ref{lem-A2} to every function $f_j$ yields
a family of functions $\zf_{1},\ldots ,\zf_{n}$. Now it is enough
setting $\zf=\zf_{1}\zpu\zpu\zpu\zf_{n}$.

Finally, if $\zf$ is not bounded take $\zf/(\zf+1)$ instead of
$\zf$.
\end{proof}



\begin{thebibliography}{99}

\bibitem{BL} M.\ Belolipetsky, A.\ Lubotzky, \emph{Finite groups and hyperbolic manifolds}, Invent.\ Math.\ \textbf{162} (2005), 459--472.

\bibitem{BW} I.\ Bumagin, D.T.\ Wise, \emph{Every group is an outer automorphism group of a finitely generated group},
J.\ Pure Appl.\ Algebra \textbf{200} (2005), 137--147.

\bibitem{CV2} C.\ Costoya, A.\ Viruel, \emph{Every finite group is the group of self homotopy equivalences of an elliptic space},
preprint arXiv:1106.1087

\bibitem{F} R.\ Frucht, \emph{Herstellung von Graphen mit vorgegebener abstrakter Gruppe}, Compositio Math.\ \textbf{6} (1938), 239--250.

\bibitem{Greenberg} L.\ Greenberg, \emph{Maximal groups and signatures}, Ann.\ Math.\ Stud., vol.\ 79, pp.\ 207--226. Princeton University Press 1974.


\bibitem{HI} M.S.\ Hirsch \emph{ Differential Topology}, GTM\ 33,
Springer 1976.

\bibitem{KN} S.\ Kobayashi, K.\ Nomizu  \emph{Foundations on Differential Geometry},
vol.\ I, Interscience Publishers 1963.

\bibitem{Kojima} S.\ Kojima, \emph{Isometry transformations of hyperbolic 3-manifolds}, Topology and its Appl.\ \textbf{29} (1988), 297--307.

\bibitem{RRO} R.\ Roussarie, \emph{Mod\`eles locaux de champs et de
formes}, Asterisque, vol.\ 30, Soci\'et\'e Math\'ematique de
France 1975.

\bibitem{SST} S.\ Sternberg, \emph{On the structure of local
homeomorphisms of euclidean $n$-spaces II}, Amer.\ J.\ Math.\
\textbf{80} (1958), 623--631.

\bibitem{WA} A.\ Wasserman, \emph{Equivariant differential
topology}, Topology\ \textbf{8} (1969), 127--150.


\end{thebibliography}
\end{document}